\documentclass[preprint,12pt,authoryear]{elsarticle}

\usepackage{amssymb}
\usepackage{amssymb}
\usepackage{epstopdf}
 \usepackage{a4wide}
  \usepackage{amsmath}
  \usepackage{amssymb}
  \usepackage{latexsym}
  \usepackage{natbib}
  \usepackage{hyperref}
 \usepackage{amsthm}

\newcommand{\E}{\mathbb E}
\newcommand{\Prob}{\mathbb{P}}

\newcommand{\Ical}[0]{\ensuremath{{\mathcal I}}}

\newcommand{\Ncal}[0]{\ensuremath{{\mathcal N}}}

\newcommand{\dtv}{d_{\text{TV}}}
\newtheorem{theorem}{Theorem}[section]
\newtheorem{lemma}[theorem]{Lemma}

\newtheorem{remark}[theorem]{Remark}

\journal{Statistics and Probability Letters}

\begin{document}

\begin{frontmatter}



\title{On the recent-$k$-record of discrete random variables}


\author{Anshui Li}

\address{School of Mathematics, Physics and Information, Shaoxing University, Shaoxing, 312000, P.R.China} 
\ead{anshuili@usx.edu.cn}
%{organization={},%Department and Organization
            %addressline={}, 
           % city={},
            %postcode={}, 
            %state={},
           % country={}}

\begin{abstract}
Let $X_1,~X_2,\cdots$ be a sequence of i.i.d random variables which are supposed to be observed in sequence. The $n$th value in the sequence is a $k-record~value$ if exactly $k$ of the first $n$ values (including $X_n$) are at least as large as it. Let ${\bf R}_k$ denote the ordered set of $k$-record values. The famous Ignatov's Theorem states that the random sets ${\bf R}_k(k=1,2,\cdots)$ are independent with common distribution. We introduce one new record named $recent-k-record$ (RkR in short) in this paper:  $X_n$ is a $j$-RkR if there are exactly $j$ values at least as large as $X_n$ in $X_{n-k},~X_{n-k+1},\cdots,~X_{n-1}$. It turns out that RkR brings many interesting problems and some novel properties such as prediction rule and Poisson approximation  which are proved in this paper. One application named "No Good Record" via the Lov{\'a}sz Local Lemma is also provided. We conclude this paper with some possible connection with scan statistics.
\end{abstract}




\end{frontmatter}


\section{Introduction}
Let $X_1,~X_2,\cdots$ be a sequence of  i.i.d random variable following the common probability mass function $\Prob(X=j)=p_j,~j\in \mathbb{Z}^+$. Suppose that these random variables are observed one by one , and $X_n$ is called a $k-record~value$ if
\[
|\{i\in \{1,2,\cdots, n\}: X_i\ge X_n\}|=k.
\] In other words, $X_n$ is one value with exactly $k$ values (including itself) as large as it in the sequence $X_1,~X_2,\cdots,X_n$. For fixed $k$,  a random ordered set ${\bf R}_k$ which includes all the $k$-record values in the sequence can be defined. In fact, the set 
\[
{\bf R}_1=\{R_1,~R_2,~R_3,\cdots\}
\]  can be regarded as the observation values  that are the largest yet seen when they appear, and one can think about the set ${\bf R}_2$ of observation values that are the {\em second} largest on their appearance, and so on.

The famous result which is called Ignatov's theorem states that not only do the sequences of $k$-record values share the same probability distribution for all $k$, but also these sequences are independent of each other. One can easily identify the ${\bf R}_k$ for given sequence observed via one technique used in the proof of the famous Ignatov's Theorem by defining a series of subsequence of the data sequence $X_1,~X_2,\cdots$, for example, see~\cite{2014ross}.  Later there are many variants and developments related to this topic, see~\cite{2002Best, 1988Ignatov, Pekoz1999Ignatov,1997On, simons2002doob, bunge2001record, aldous1999longest}.

In this paper, we will introduce one novel random
variable called recent-k-record (RkR in short) for some fixed integer $k\ge 1$: instead of considering the whole past story, we only consider the values of $X_{n-k},~X_{n-k+1},\cdots,~X_{n-1}$, i.e., the $k$ values before $X_n$(do not include itself). And let us define $X_n$ be a $j$-RkR if there are exactly $j$ values at least as large as $X_n$ in $X_{n-k},~X_{n-k+1},\cdots,~X_{n-1}$. In other words,  $X_n$ a $j$-RkR if
\[
|p: X_{n-p}\ge X_n, 1\le p\le k|=j.
\]  We will denote that $i \in R_j^k$ if $i$ is a $j$-RkR. In other words, there exists a subsequence with length $k+1$ such that $X_{n_0}=i$ and
\[
|p: X_{n_0-p}\ge i, 1\le p\le k|=j.
\] for some $n_0\ge k+1$. We can consider the usual $j$-record as one "dynamic version" of $j$-RkR, i.e., $k=n$ for $X_n$ in that case.

Actually, RkR can be found applications in many areas: for example, to assess one athlete's recent condition and achievements, one proper way is to  check the results in his recent records and not necessary to get the whole story( it may be nothing with his records ten years or even five years before).
For the k-records application in statistics for athletes, see~\cite{1997statistics}.

The remainder of this paper is organized as follows. In Section~\ref{sec:2}, we calculate the prediction probability for RkR. The Poisson approximation for RkR and one interesting application via the Lov{\'a}sz Local Lemma will be presented in Section~\ref{sec:3} and Section~\ref{sec:4} respectively. We conclude this paper with some possible extensions of RkR related to scan statistics in Section~\ref{sec:5}.
\section{Prediction probability for RkR}\label{sec:2}
%\section{Recent k-record}

Let us fix $k\ge 1$ and $ 0\le j\le k, n\ge k+1$ in the whole paper, denote $S_i=\Prob(X\ge i)=\sum_{s\ge i}p_s$ and $C_i=\Prob(X\le i)=\sum_{j=1}^ip_j$. It is obvious that $S_i+C_{i-1}=1$.
\begin{theorem} It is easy to have the following observations
\begin{enumerate}
\item $\Prob(i \in R_j^k~in~(X_1,X_2,\cdots,X_{k+1}))=\binom{k}{j}S_i^j(C_{i-1})^{k-j}p_i$;
\item $\Prob(i \in R_{j+1}^k ~in~(X_2,X_3,\cdots,X_{k+2})|i \in R_j^k~in~(X_1,X_2,\cdots,X_{k+1}))=(C_{i-1})p_i$;
\item $\Prob(i \in R_{j}^k ~in~(X_2,X_3,\cdots,X_{k+2})|i \in R_j^k~in~(X_1,X_2,\cdots,X_{k+1}))=S_ip_i$;
\item $\Prob(i\notin R_j^k~ \text{in}~ (X_1,X_2,\cdots,X_{k+1} ))=\left(1-\binom{k}{j}S_i^j(C_{i-1})^{k-j}\right)p_i+(1-p_i)$;

\item The probability $A$ that $(i \in R_j^k)$ is in the first n-sequence is at most 
\[
\Prob(A)\le (n-k)\binom{k}{j}S_i^j(C_{i-1})^{k-j}p_i.
\]
 
%\item The expected number $Y$ of $(i \in R_j^k)$ in the first n sequence is
%\[
%\E(Y)=(n-k)\binom{k}{j}S_i^j(C_{i-1})^{k-j}p_i
%\]
\end{enumerate}
\end{theorem}

\begin{theorem}\label{thm:dis}
%$\Prob(i\in R_j^k)=\binom{k}{j}(1-\sum_{1}^{i-1}p_m)^{j}(\sum_{1}^{i-1}p_m)^{k-j}p_i$
$\Prob(X_n\in R_j^k)= \sum_{i}\binom{k}{j}S_i^{j}(C_{i-1})^{k-j}p_i$, as a result, we have
\[
q_i=\Prob(X_n=i|X_n\in R_j^k)=\frac{\binom{k}{j}S_i^j(C_{i-1})^{k-j}p_i}{\sum_{l}\binom{k}{j}S_l^{j}(C_{l-1})^{k-j}p_l}
\]
\end{theorem}
\begin{proof} The result is easy to get by conditioning on $X_n$,
\begin{equation}
\begin{aligned}
\Prob(X_n \in R_j^k) &=\sum_{l}\Prob(X_n\in R_j^k| X_n=l)\Prob(X_n=l)\\
                              &=\sum_l\binom{k}{j}S_l^j(C_{l-1})^{k-j}p_l\\
                              &=\binom{k}{j}\sum_iS_l^j(C_{l-1})^{k-j}p_l
\end{aligned}
\end{equation}
And the second result is obtained by Bayes' rule.
\end{proof}
From Theorem~\ref{thm:dis}, we can assert that the $k$ random variables 
\[
R_1^k, R_2^k,\cdots,R_k^k
\]
donot have the same distribution and of course they are not independent either. In other words, our result here is completely different with the famous Ignatov's Theorem.

In the following result, we will predict $X_{n+1}$ based on the states of $X_n$ for $n\ge k+1$.
\begin{theorem}[Prediction rule]
$$\Prob(X_{n+1}\in R_j^k| X_n\in R_j^k)=\sum_{i}q_i\left(S_ip_i+p_m\left(\sum_{m>i}\left(\frac{S_m}{S_i}\right)^j\frac{k-j}{k}+\sum_{m<i}\left(\frac{C_{m-1}}{C_{i-1}}\right)^{k-j}\frac{j}{k}\right)\right).$$
\end{theorem}
\begin{proof}
\begin{equation}\label{equ:con}
\begin{aligned}
\Prob(X_{n+1}\in R_j^k| X_n\in R_j^k)\\
&=\Prob(X_{n+1}\in R_j^k,X_n=X_{n+1}| X_n\in R_j^k)+\Prob(X_{n+1}\in R_j^k,X_n\neq X_{n+1}| X_n\in R_j^k)
\end{aligned}
\end{equation}
Then we use the following formula of conditional probability
\[
\Prob(A|B)=\sum_{C_i}\Prob(AC_i|B)=\sum_{C_i}\Prob(A|C_iB)\Prob(C_i|B)
\] in which $\{C_i\}_{i\ge 1}$ is a partition of the corresponding probability space $\Omega$.

Then the equation (\ref{equ:con}) can be written by letting $\Omega=\sum_{i}(X_n=i)$
\begin{equation}
\begin{aligned}
&=\sum_{i}\Prob(X_{n+1}\in R_j^k,X_n=X_{n+1},X_n=i| X_n\in R_j^k)+\sum_i\Prob(X_{n+1}\in R_j^k,X_n\neq X_{n+1},X_n=i| X_n\in R_j^k)\\
&=\sum_{i}\Prob(X_{n+1}\in R_j^k,X_n=X_{n+1}| X_n=i,X_n\in R_j^k)\Prob(X_n=i|X_n\in R_j^k)\\
&+\sum_i\Prob(X_{n+1}\in R_j^k,X_n\neq X_{n+1}| X_n=i,X_n\in R_j^k)\Prob(X_n=i|X_n\in R_j^k)\\
&=\sum_{i}\Prob(X_{n+1}=i\in  R_j^k|X_n=i\in R_j^k)q_i
+\sum_i\Prob(X_{n+1}\in R_j^k,X_n\neq X_{n+1}|X_n= i\in R_j^k)q_i\\
&=\sum_{i}\Prob(X_{n+1}=i\in  R_j^k|X_n=i\in R_j^k)q_i+\sum_i\sum_{m\neq i}\Prob(X_{n+1}=m\in R_j^k|X_n= i\in R_j^k)q_i
\end{aligned}
\end{equation}
Actually, the first part of the formula above is easy and we have 
\begin{equation}
\begin{aligned}
\sum_{i}\Prob(X_{n+1}=i\in  R_j^k|X_n=i\in R_j^k)\Prob(X_n=i|X_n\in R_j^k)&=
\sum_{i}S_ip_iq_i.
\end{aligned}
\end{equation}
Then we will analyze the second part in several steps as follows:
\begin{equation}
\begin{aligned}
\sum_{m\neq i}\Prob(X_{n+1}=m\in R_j^k|X_n= i\in R_j^k)
&=\sum_{m>i}\Prob(X_{n+1}=m\in R_j^k|X_n= i\in R_j^k)\\
&+\sum_{m<i}\Prob(X_{n+1}=m\in R_j^k|X_n= i\in R_j^k)
\end{aligned}
\end{equation}
Then we discuss the two different cases accordingly:
\begin{enumerate}
\item For the case $m>i$: we have 
\begin{equation}
\begin{aligned}
\sum_{m>i}\Prob(X_{n+1}=m\in R_j^k|X_n= i\in R_j^k)
&=\sum_{m>i}\Prob(X_{n+1}=m\in R_j^k,X_{n-k}<i|X_n= i\in R_j^k)\\
&+\sum_{m>i}\Prob(X_{n+1}=m\in R_j^k,X_{n-k}\ge i|X_n= i\in R_j^k)
\end{aligned}
\end{equation}

\begin{enumerate}
\item Conditioning on the event $\{X_{n-k}< i\}$: the event $\{X_n=i\in R_j^k\}$ indicates that there are exactly $j$ elements which are at least as large as $X_n=i$ in $X_{n-k+1},X_{n-k+2},\cdots,X_{n-1}$;
 the event $\{X_{n+1}=m\in R_j^k\}$ indicates that there are exactly $j$ elements which are at least as large as $X_{n+1}=m$ in $X_{n-k+1},X_{n-k+2},\cdots,X_{n-1}$. To sum up: the event $(X_{n+1}=m\in R_j^k,X_{n-k}<i,X_n= i\in R_j^k)$ means: there are  $j$ elements which are at least as large as $X_{n+1}=m$ and $k-1-j$ elements which are strictly less than $i$ in $X_{n-k+1},X_{n-k+2},\cdots,X_{n-1}$, and $X_{n-k}<i, X_n=i, X_{n+1}=m$. I.e.,
\begin{equation}
\begin{aligned}
 \Prob(X_{n+1}=m\in R_j^k, X_{n-k}<i|X_n= i\in R_j^k)
 &=\frac{\Prob(X_{n+1}=m\in R_j^k, X_{n-k}<i,X_n= i\in R_j^k)}{\Prob(X_n= i\in R_j^k)}\\
 &=\frac{\binom{k-1}{j}S_m^j(C_{i-1})^{k-1-j}C_{i-1}p_ip_m}{\binom{k}{j}S_i^j(C_{i-1})^{k-j}p_i}\\
 &=\left(\frac{S_m}{S_i}\right)^j\left(\frac{k-j}{k}\right)p_m
\end{aligned}
\end{equation}
\item Conditioning on the event $\{X_{n-k}\ge i\}$: the event $\{X_n=i\in R_j^k\}$ indicates that there are exactly $j-1$ elements which are at least as large as $X_n=i$ in $X_{n-k+1},X_{n-k+2},\cdots,X_{n-1}$;
 the event $\{X_{n+1}=m\in R_j^k\}$ and $X_n=i<m$ indicates that there are exactly $j$ elements which are at least as large as $X_{n+1}=m>i$ in $X_{n-k+1},X_{n-k+2},\cdots,X_{n-1}$. This is not possible, so the probability is zero.
 \end{enumerate}
\item For the case $m<i$:we have 
\begin{equation}
\begin{aligned}
\sum_{m<i}\Prob(X_{n+1}=m\in R_j^k|X_n= i\in R_j^k)
&=\sum_{m<i}\Prob(X_{n+1}=m\in R_j^k,X_{n-k}<i|X_n= i\in R_j^k)\\
&+\sum_{m<i}\Prob(X_{n+1}=m\in R_j^k,X_{n-k}\ge i|X_n= i\in R_j^k)
\end{aligned}
\end{equation}

\begin{enumerate}
\item  Conditioning on the event $\{X_{n-k}<i\}$: the event $\{X_n=i\in R_j^k\}$ indicates that there are exactly $j$ elements which are at least as large as $X_n=i$ in $X_{n-k+1},X_{n-k+2},\cdots,X_{n-1}$, which means there will be $j+1$ elements which are as large as $m$ in $X_{n-k+1},X_{n-k+2},\cdots,X_n$, contradicting the event $\{X_{n+1}=m\in R_j^k\}$.
\item Conditioning on the event $\{X_{n-k}\ge i\}$: the event $\{X_n=i\in R_j^k\}$ indicates that there are exactly $j-1$ elements which are at least as large as $X_n=i$ in $X_{n-k+1},X_{n-k+2},\cdots,X_{n-1}$;
 the event $\{X_{n+1}=m\in R_j^k\}$ indicates that there are exactly $j-1$ elements which are at least as large as $X_{n+1}=m$ in $X_{n-k+1},X_{n-k+2},\cdots,X_{n-1}$. To sum up: the event $(X_{n+1}=m\in R_j^k,X_{n-k}<i,X_n= i\in R_j^k)$ means: there are  $j-1$ elements which are at least as large as $X_{n}=i>m$ and $k-1-j$ elements which are strictly less than $m$ in $X_{n-k+1},X_{n-k+2},\cdots,X_{n-1}$, and $X_{n-k}\ge i, X_n=i, X_{n+1}=m$. I.e.,
 \begin{equation}
\begin{aligned}
 \Prob(X_{n+1}=m\in R_j^k, X_{n-k}\ge i|X_n= i\in R_j^k)
 &=\frac{\Prob(X_{n+1}=m\in R_j^k, X_{n-k}\ge i,X_n= i\in R_j^k)}{\Prob(X_n= i\in R_j^k)}\\
 &=\frac{\binom{k-1}{j-1}S_i^{j-1}(C_{m-1})^{k-j}S_ip_ip_m}{\binom{k}{j}S_i^j(C_{i-1})^{k-j}p_i}\\
 &=\left(\frac{C_{m-1}}{C_{i-1}}\right)^{k-j}\left(\frac{j}{k}\right)p_m
\end{aligned}
\end{equation}

\end{enumerate}
\end{enumerate}
Finally, we put all the pieces together, we can have 
\begin{equation}
\begin{aligned}
\Prob(X_{n+1}\in R_j^k| X_n\in R_j^k)
&=\sum_{i}\Prob(X_{n+1}\in R_j^k,X_n=X_{n+1},X_n=i| X_n\in R_j^k)\\
&+\sum_i\Prob(X_{n+1}\in R_j^k,X_n\neq X_{n+1},X_n=i| X_n\in R_j^k)\\
&=\sum_{i}q_i\left(S_ip_i+p_m\left(\sum_{m>i}\left(\frac{S_m}{S_i}\right)^j\frac{k-j}{k}+\sum_{m<i}\left(\frac{C_{m-1}}{C_{i-1}}\right)^{k-j}\frac{j}{k}\right)\right).
\end{aligned}
\end{equation}
\end{proof}
\begin{remark}
Actually, we can get the more general case of the prediction probability
\[
\Prob(X_{n+1}\in R_{j_1}^k| X_n\in R_{j_2}^k)
\] for  $j_1\neq j_2$ by the same argument above.
\end{remark}
%\begin{theorem}
%$\Prob(i\notin R_j^k~ \text{in}~ (X_1,X_2,\cdots,X_{k+1} ))=\left(1-\binom{k}{j}S_i^j(1-S_i)^{k-j}\right)p_i+(1-p_i).$
%\end{theorem}
%\begin{proof}
%The result is easy to get by conditioning on the $X_{k+1}$. 
%\begin{equation}
%\begin{aligned}
%\Prob(i\notin R_j^k~ in~ (X_1,X_2,\cdots,X_{k+1} ))&=\Prob(i\notin R_j^k,~X_{k+1}=i)+\Prob(i\notin R_j^k,~X_{k+1}\neq i)\\
%&=\Prob(i\notin R_j^k|X_{k+1}=i)\Prob(X_{k+1}=i)+\Prob(i\notin R_j^k|X_{k+1}\neq i)\Prob(X_{k+1}\neq i)\\
%&=\left(1-\binom{k}{j}S_i^j(1-S_i)^{k-j}\right)p_i+(1-p_i)
%\end{aligned}
%\end{equation}
%\end{proof}
%\begin{theorem}
%$\Prob(i\in R_j^k)=p_i.$
%\end{theorem}
%\begin{proof}
%Let $B=\{i\in R_j^k\}$, and $B_{s}=\{X_{k+s}=i\}$, it is easy to see that
%\[
%B\subsetneq \cup_{s\ge 1}B_s.
%\]
%\end{proof}
\section{Poisson approximation for $R_j^k$}\label{sec:3}
The asymptotic properties of sum of random variables are very important in probability and statistics. It is well known that convergence to a Poisson distribution can occur if the individual  means of Bernoulli random variable are all small even if they are not independent, more detained information can be found in ~\cite{dasgupta2008asymptotic}. In this section, we will give the Poisson approximation for $R_j^k$ using the Stein-Chen method. 
%We first define a series of random variable and give some basic properties of them.

We will give the definition of dependency graph first and then give the Poisson approximation Lemma based the dependency graph.

\subsection{Dependency graph in general  and Poisson approximation Lemma}
 
Let $(I,E)$ be a graph with finite or countable vertex set $I$ and edge set $E$. For $i,j\in I$, we denote $i\sim j$ if $(i,j)\in E$. For $i\in I$, let $\Ncal_i=\{i\}\cup \{j\in I: i\sim j\}$. The graph $(I,\sim)$ is called a dependency graph for a collection of random variables $(\xi_i,i\in I)$ if for any two disjoint subsets $I_1,I_2$ of $I$ such that there are no edges connecting $I_1$ to $I_2$, the collection of random variables $\{\xi_i, i\in I_1)$ is independent of $\{\xi_i, i\in I_2)$. The notion of dependency graphs gives a very useful to express some rare-independence, which is a technique to generalize the independence.

 The Lemma below gives the total variance of two distributions by Stein-Chen technique with the help of the dependency graphs.
\begin{lemma}[\cite{1989Two}]\label{lem:stein}
Suppose $(\xi_i,i\in I)$ is a finite collection of Bernoulli random variables with dependency graph $(I,\sim)$. Set $p_i:=\Prob(\xi_i=1)=\E(\xi_i)$, and set $p_{ij}:=\E(\xi_i\xi_j)$. Let $\lambda:=\sum_{i\in I}p_i$, and suppose $\lambda$ is finite, let $W:=\sum_{i\in I}\xi_i$. Then
\[
\dtv(W,Po(\lambda))\le \min(3,\lambda^{-1})\left(\sum_{i\in \Ical}\sum_{j\in \Ncal(i)\setminus \{i\}}p_{ij}+\sum_{i\in \Ical}\sum_{j\in \Ncal(i)}p_ip_j\right).
\]
In which $\dtv(\xi,\eta)=\sup_{A\subseteq \mathbb{Z}}|\Prob(\xi\in A)-\Prob(\eta\in A)|$ for two integer-valued random variables $\xi,\eta$ and $Po(\lambda)$ is the Poisson distribution with parameter $\lambda$.
\end{lemma}

\subsection{Poisson approximation for RkR}
For fixed $i_0\in \mathbb{Z}^+$, we define a series of random variables as follows:
\[
\xi_ i=\left\{
\begin{array}{rl}
1 &\text{if} ~i_0\in R_j^k ~~\text{in}~~ (X_i,X_{i+1},\cdots,X_{i+k}),\\
0 & \text{else}.
\end{array}\right.
\]
Which means the random variables $\xi_i$ are indexed by the $(k+1)$-set $\{X_i,X_{i+1},\cdots,X_{i+k}\}$.

There are many interesting properties on these random variables $\xi_i$.
\begin{theorem}\label{thm:rv} For fixed $i_0\in \mathbb{Z}^+$, we have the following results:
\begin{enumerate}
\item $\E(\xi_i)=\Prob(\xi_i=1)=\binom{k}{j}S_{i_0}^j(1-S_{i_0})^{k-j}p_{i_0},$
\item $\E(\xi_i\xi_{i+1})=\binom{k-1}{j-1}S_{i_0}^{j}(1-S_{i_0})^{k-j}p_{i_0}^2$,
%\item For $|(X_{i_1},X_{i_1+1},\cdots,X_{i_1+k})\cap (X_{i_2},X_{i_2+1},\cdots,X_{i_2+k})|=m,~\E(\xi_{i_1}\xi_{i_2})$

\item For $|i_1-i_2|=m > k,$
\[
~\Prob(\xi_{i_1}=1,\xi_{i_2}=1)=\Prob(\xi_{i_1}=1)\Prob(\xi_{i_2}=1),~~~~\E(\xi_{i_1}\xi_{i_2})=\E(\xi_{i_1})\E(\xi_{i_2}),
\]
\item For $|i_1-i_2|=m\in\{1,2,\cdots,k\}$, 
\begin{footnotesize}
\begin{equation}
\begin{aligned}
\phi_m&=E(\xi_{i_1}\xi_{i_2})\\
& =\sum_{t=\max\{0,j-m-1\}}^{\min\{k-m,j-1\}}\binom{m}{j-t}S_{i_0}^{j-t}(1-S_{i_0})^{m-j+t}\binom{m-1}{j-t-1}S_{i_0}^{j-t-1}(1-S_{i_0})^{m-j+t}\binom{k-m}{t}S_{i_0}^t(1-S_{i_0})^{k-m-t}p_{i_0}^2\\
&=\sum_{t=\max\{0,j-m-1\}}^{\min\{k-m,j-1\}}\binom{m}{j-t}\binom{m-1}{j-t-1}\binom{k-m}{t}S_{i_0}^{2j-t-1}(1-S_{i_0})^{m-2j+t+k}p_{i_0}^2
%&=\sum_{t=\max\{0,j-m-1\}}^{\min\{k-m,j-1\}}\frac{m_{(j-t)}m_{(j-t-1)}(k-m)_{(t)}}{m(j-t)!(j-t-1)!t!}S_{i_0}^{2j-t-1}(1-S_{i_0})^{m-2j+t+k}p_{i_0}^2.
\end{aligned}
\end{equation}
\end{footnotesize}
\end{enumerate}
\end{theorem}

 We then define the  dependency graph for RkR as follows:
 
 Let $\mathcal{I}_n$ be the set of all $(k+1)$-sets  $\{X_i,X_{i+1} \cdots,X_{i+k}\}$ of $\{X_1,X_2,\cdots,X_{n+k}\}$. It is easy to see that
 the size of $\mathcal{I}_n$ is $n$. For each element $\mathbf{i}\in \mathcal{I}_n$, let $\mathcal{N}_{i}$ be the set of $\mathbf{j}\in \mathcal{I}_n$ such that $\mathbf{i}$ and $\mathbf{j}$ have at least one element in common. And let $\mathbf{i}\sim \mathbf{j}$ if $\mathbf{j}\in \mathcal{N}_i$ but $\mathbf{i}\neq \mathbf{j}$. In other words, $\mathcal{N}_{i}=\{\mathbf{i}\}\cup \{\mathbf{j}\in\mathcal{I}_n:\mathbf{i}\sim \mathbf{j}\}$.Then $\xi_{i}$ is independent of $\xi_j$ except when $\mathbf{j}\in \mathcal{N}_i$, and as a result the graph $(\mathcal{I}_n,\sim)$ is a dependency graph for $\xi_i,~i=1,2,\cdots,n$.

As a consequence of Lemma~\ref{lem:stein} and Theorem~\ref{thm:rv}, we can get the following result easily:

\begin{theorem}
Let $i_0\in \mathbb{Z}^+$ be fixed, then the number of $i_0\in R_j^k$ in the sequence $(X_1,X_2,\cdots,X_{n+k})$ is  $\xi=\sum_{i=1}^n\xi_i$, which has an asymptotic Poisson distribution with parameter $\lambda:=n\binom{k}{j}S_{i_0}^j(1-S_{i_0})^{k-j}p_{i_0}$. To be more precise, we have 
\[
\dtv(\xi,Po(\lambda))\le n\min\{3,\lambda^{-1}\}\left(\sum_{s=1}^k\phi_s+(k+1)p^2\right).
\] In which $\phi_s=\E(\xi_{i_1}\xi_{i_2})$ when $|i_1-i_2|=s\le k$ and $p=\binom{k}{j}S_{i_0}^j(1-S_{i_0})^{k-j}p_{i_0}$.
\end{theorem}
\begin{proof}
It is easy to get that $\Prob(\xi_i=1)=\binom{k}{j}S_{i_0}^j(1-S_{i_0})^{k-j}p_{i_0}=p$, leading to 
\[
\lambda=\E \xi =\sum_{i=1}^n \E(\xi_i)=np.
\]
We then get
\[
\sum_{i\in \mathcal{I}_n}\sum_{j\in \mathcal{N}(i)\setminus \{i\}}p_{ij}=\sum_{i\in \mathcal{I}_n}\sum_{s=1}^k\phi_s=n\sum_{s=1}^k\phi_s
\]
and
\[
\sum_{i\in \mathcal{I}_n}\sum_{j\in \mathcal{N}(i)}p_ip_j=n\sum_{j\in \mathcal{N}(i)}\E(\xi_i=1)\E(\xi_j=1)=n(k+1)p^2.
\]
Then by Lemma~\ref{lem:stein}, we complete the proof.
\end{proof}

\section{"No Good Record" via the Lov{\'a}sz Local Lemma}\label{sec:4}
Let $E_i$ be the event that $i_0\in R_j^k$ in $(X_i,X_{i+1},\cdots,X_{i+k}),(i=1,2,\cdots,n)$. In this section, we will show that there are positive probability that $i_0$ will not be one RkR in the sequence $\{X_1,X_2,\cdots,X_{n+k}\}$, i.e., the events $\cap_{i=1}^n \bar{E}_i$ can happen with positive probability once $p_{i_0}$ is chosen properly. Our result bases mainly on one version of the famous Lov{\'a}sz Local Lemma which can be checked in ~\cite{2005Probability}.
\begin{lemma}[Lov{\'a}sz Local Lemma]
Let $E_1,\cdots, E_n$ be a set of events, and assume that the following hold:
\begin{enumerate}
\item for all $i,~\Prob(E_i)\le p$;
\item the degree of the dependency graph given by $E_1,\cdots,~E_n$ is bounded by $d$;
\item $4dp\le 1$.
\end{enumerate}
Then 
\[
\Prob(\cap_{i=1}^n \bar{E}_i)>0.
\]
\end{lemma}

Our result goes as follows:
\begin{theorem}["No Good Result" Theorem]
Let $E_i$ be the event that $i_0\in R_j^k$ in $(X_i,X_{i+1},\cdots,X_{i+k}),$
$(i=1,2,\cdots,n)$. There exists some $p_{i_0}>0$,  such that
\[
\Prob(\cap_{i=1}^n \bar{E}_i)>0.
\]In other words, there exists some one with "no good record" in the whole story with positive probability.
\end{theorem}
\begin{proof} 
Let $p_i=\E(E_i)$, then we have 
\begin{equation}
\begin{aligned}
4kp_i &=4k\binom{k}{j}S_{i_0}^j(1-S_{i_0})^{k-j}p_{i_0}\\
        &\le 4k\left(\frac{ke}{j}\right)^jS_{i_0}^j(1-S_{i_0})^{k-j}p_{i_o}\\
        &=4k\left(\frac{ke}{j}\right)^j\left(\frac{jS_{i_0}+(k-j)(1-S_{i_0})}{k}\right)^kp_{i_0}\\
        &=4k\left(\frac{ke}{j}\right)^j \left(\frac{k-j+S_{i_0}(2j-k)}{k}\right)^kp_{i_0}\\
        &\le 4k\left(\frac{ke}{j}\right)^j \max\left\{\left(\frac{k-j}{k}\right)^k,\left(\frac{j}{k}\right)^k\right\}p_{i_0}.\\
        &:=C(k,j)p_{i_0}.
\end{aligned}
\end{equation}
Since $C(k,j)$ is some constant depending on $k,j$, one can choose $p_{i_0}$ accordingly to make sure 
       \[
       C(k,j)p_{i_0}<1.
       \]
\end{proof}

\section{Further Extensions: RkR and Scan Statistics}\label{sec:5}

We conclude our paper with some possible extensions: there are some naturally connection between the RkR proposed in this paper and the scan statistics.

Scan statistics(~\cite{glaz2001scan,marchette2012scan,kulldorff1999spatial}) used in many areas mainly measure an unusually large cluster of events in time or space. Formally speaking, given events distributed over a time period $(0,T)$, $S_w$ is the largest number of events in any subinterval of length $w$. Then $S_w$ is called the {\em scan statistic}, which can be regarded that one scans the time period $(0,T)$ with a window of size $w$, and identifies the maximum cluster of points. 

It turns out that some property of RkR can be obtained with the help of scan statistics: we will consider the maximum number of values that exceed some (fixed or random) value respectively.
\subsection{The fixed target}
 If we consider some fixed $m\in \mathbb{Z}^+$, i.e., $m$ is some possible value that $X$ can take, and then we define 
 \[
\eta_ i=\left\{
\begin{array}{rl}
1 &\text{if} ~~~X_i\ge m,\\
0 & \text{else}.
\end{array}\right.
\]
Then we can set 
\[
Y_s=\sum_{i=s}^{i=s+k-1}\eta_i.
\]
It is easy to see that 
\[
S_k=\max_{s}\{Y_s\}
\]
is one scan statistic which scans the time period $(0,T)$ with window size $k$.

\subsection{The random target}
If  we define 
 \[
\eta_{ij}=\left\{
\begin{array}{rl}
1 &\text{if} ~~~X_j\ge X_i,\\
0 & \text{else}.
\end{array}\right.
\]
Then can set 
\[
Y_i=\sum_{j=i-k}^{j=i-1}\eta_{ij}.
\]
It is easy to see that 
\[
S_k=\max_{i}\{Y_i\}
\]
is one scan statistic which scans the time period $(0,T)$ with window size $k$.

We can get some results on $S_k$ with the help of the estimation of scan statistics, see \cite{naus1974probabilities} for some general result of $\Prob(S_k\ge s)$ based on some random model in which there are $N$ successes distributed at random over the $T$ trials. Of course, some asymptotic results can be derived when $T$ goes to $\infty$ and $N/T$ may vary in some way.




\section*{Acknowledgments}
%We thank the anonymous referee for careful proofreading and suggestions that have improved the presentation. 
A.Li's research was supported by Natural Science Fund of China(grant number 11901145).
\section*{Reference}
%\begin{thebibliography}



%\end{thebibliography}
\bibliographystyle{elsarticle-harv.bst}
\bibliography{spl}
%\end{thebibliography}
\end{document}